\documentclass[10pt,a4paper]{amsart}
\usepackage[final]{optional}
\usepackage[british,american]{babel}

\usepackage{mathrsfs}
\usepackage{amsfonts, amsmath, wasysym}
\usepackage{amssymb,amsthm,
paralist
}
\usepackage{
latexsym,
nicefrac}

\usepackage[usenames]{color}

\usepackage{url}

\definecolor{darkgreen}{rgb}{0,0.5,0}
\definecolor{darkred}{rgb}{0.7,0,0}
\usepackage[colorlinks, 
citecolor=darkgreen, linkcolor=darkred
]{hyperref}
\usepackage{esint}
\usepackage{bibgerm}
\usepackage[normalem]{ulem}

\theoremstyle{plain}
\newtheorem{lemma}{Lemma}[section]
\newtheorem{thm}[lemma]{Theorem}

\newtheorem{cor}[lemma]{Corollary}

\theoremstyle{definition}

\newtheorem{rmk}[lemma]{Remark}
\numberwithin{equation}{section}

\newcommand{\al}{\alpha}

\newcommand{\de}{\delta}
\newcommand{\om}{\omega}

\newcommand{\La}{\Lambda}

\newcommand{\si}{\sigma}

\newcommand{\Si}{\Sigma}
\renewcommand{\th}{\theta}

\newcommand{\R}{\ensuremath{{\mathbb R}}}
\newcommand{\N}{\ensuremath{{\mathbb N}}}

\newcommand{\C}{\ensuremath{{\mathbb C}}}

\newcommand{\Cyl}{{\mathscr{C}}}

\DeclareMathOperator{\inj}{inj}

\newcommand{\norm}[1]{\Vert#1\Vert}  

\def\osc{\mathop{{\mathrm{osc}}}\limits}

\newcommand{\arsinh}{{\rm arsinh}}

\newcommand{\beq}{\begin{equation}}
\newcommand{\eeq}{\end{equation}}
\newcommand{\beqs}{\begin{equation*}}
\newcommand{\eeqs}{\end{equation*}}
\newcommand{\beqa}{\begin{equation}\begin{aligned}}
\newcommand{\eeqa}{\end{aligned}\end{equation}}
\newcommand{\beqas}{\begin{equation*}\begin{aligned}}
\newcommand{\eeqas}{\end{aligned}\end{equation*}}
\newcommand{\brmk}{\begin{rmk}}
\newcommand{\ermk}{\end{rmk}}
\newcommand{\partref}[1]{\hbox{(\csname @roman\endcsname{\ref{#1}})}}
\newcommand{\half}{\frac{1}{2}}

\newcommand{\gi}{{g_{\ell_i}}}

\newcommand{\A}{\ensuremath{{\mathcal A}}}

\newcommand{\abs}[1]{\vert#1\vert} 
\newcommand{\babs}[1]{\left\vert#1\right\vert}
 
\newcommand{\eps}{\varepsilon}
\newcommand{\na}{\nabla}

\newcommand{\Col}{{\mathcal C}}

\newcommand{\bit}{\begin{itemize}}
\newcommand{\eit}{\end{itemize}}

\newcommand{\la}{\lambda}

\newcommand{\dist}{\text{dist}}

\newcommand{\tauE}{\tau_{g_E}}

\newcommand{\hatumi}{{\hat u_{i}}}
\newcommand{\hatui}{{\hat u_i}}
\newcommand{\phisn}{\phi}

\title{{\sc
Convergence of almost harmonic maps to geodesic bubble trees
}}
\author{Melanie Rupflin}
\date{\today}
\begin{document}

\begin{abstract}

We prove a sharp criterion on the decay of the tension of almost harmonic maps from degenerating surfaces that ensures that such maps subconverge to a limiting object that is made up entirely of harmonic maps. 

\end{abstract}

\maketitle

\section{Introduction}\label{sect:intro}
A map $u$ from a closed Riemannian surface $(M,g)$ to a Riemannian manifold $(N,g_N)$ is called harmonic if it is a critical point of the Dirichlet energy $$E(u,g)=\half\int_M\abs{ du}_g^2dv_g.$$ 
Harmonic maps are characterised by $\tau_g(u)=0$, where,  viewing $(N,g_N)$ as isometrically embedded in some $\R^N$ and writing
 $A$ for the second fundamental form, the tension field $\tau_g(u)$
is given by
 $$\tau_g(u)=\Delta_gu+A_g(u)(d u,d u)=\Delta_g u+g^{ij}A(u)(u_{x_i},u_{x_j}).$$
For any \textit{fixed} domain surface $(M,g)$
the compactness results from  \cite{Struwe1985, DT, QingTian, LinWang} ensure that for maps $u_i:(M,g)\to N$ with bounded energy
which are almost harmonic in the sense that 
\beq \label{def:almost-harm-fixed-domain}
\norm{\tau_g(u_i)}_{L^2(M,g)}\to 0 \text{ as } i\to \infty\eeq
a subsequence converges to  a bubble tree that  consists of a harmonic base map $u_\infty:M\to N$ and a finite number of harmonic spheres $\om_i:S^2\to N$ that bubble off and that in this convergence there is no loss of energy nor formation of non-trivial necks. Hence for almost harmonic maps from any fixed domain surface the  limit against which the sequence converges is made up entirely of objects that are themselves critical points of the Dirichlet energy, i.e. harmonic maps.

It is now natural to ask whether the same can be expected if we instead consider a sequence of maps $u_i:M\to N$ which are almost harmonic with respect to a sequence of metrics $g_i$ on $M$. Unsurprisingly the answer is positive if the metrics themselves converge to a limiting metric 
$g_\infty$, after potentially pulling-back by suitable diffeomorphisms. In particular, if we use the uniformisation theorem and the conformal invariance of the Dirichlet energy to restrict our attention to domain metrics of constant curvature $\kappa_{g_i}=1,0,-1$ for surfaces of genus $0,1, \geq 2$ (with unit area if $\gamma=1$) then this holds true  unless the injectivity radius of $(M,g_i)$ tends to zero.

If $\inj(M,g_i)\to 0$ then the situation is more involved as parts of the surface degenerate. In this setting the compactness properties of \textit{harmonic maps} were investigated in \cite{Chen-Tian-main, Li-Wang-main, Miaomiao}, while the compactness properties of almost harmonic maps were considered in \cite{RTZ,HRT}. We note that such almost harmonic maps from degenerating surfaces in particular arise in the asymptotic analysis of Teichm\"uller harmonic map flow, see \cite{RT,RTZ,HRT}.

To describe these results we first recall some well known properties about degenerating hyperbolic surfaces, see \cite{Tromba} for more details. 

So let  $(M,g_i)$ be a sequence of surfaces with Gauss-curvature $\kappa_{g_i}\equiv -1$ for which $\inj(M,g_i)\to 0$. Then we can pass to a subsequence so that for some $k\in\{1,\ldots,3(\gamma-1)\}$ there are simple closed geodesics $\{\si_i^j\}_{j=1}^k$ of length $\ell_i^j\to 0$, a complete hyperbolic surface $(\Si,h)$ with $2k$ punctures and diffeomorphisms $f_i:\Si\to M\setminus \bigcup_{j=1}^k \si_i^j$ so that 
 $$f_i^*g_i\to h \text{ smoothly locally on } \Si.$$
Furthermore we know that the degenerating parts of the surface $(M,g_i)$ are contained in the union of so called collar neighbourhoods $\Col(\si_i^j)$ around the collapsing geodesics which are isometric to hyperbolic cylinders $(\Col(\ell_i^j),g_{\ell_i^j})$ characterised by
\beq 
\label{def:collar}
(\Col(\ell), g_\ell)=\big([-X(\ell),X(\ell)]\times S^1, \rho_{\ell}
 ^2(s)(ds^2+d\th^2)\big)
 \eeq
and
\beq
\label{def:X-rho}
X(\ell)=\tfrac{2\pi}{\ell}\left(\tfrac\pi2-\arctan\left(\sinh\left(\tfrac{\ell}{2}\right)\right) \right)
\text{ and } \rho_\ell(s)=\tfrac{\ell}{2\pi\cos(\tfrac{\ell}{2\pi} s)}.\eeq
If we now consider a  sequence of maps $u_i:M\to N$ from such degenerating hyperbolic surfaces which satisfy  $\norm{\tau_{g_i}(u_i)}_{L^2(M,g_i)}\to 0$, then 
on compact subsets of $\Si$  the almost harmonic maps $u_i\circ f_i:(\Si,f_i^*g_i)\to (N,g_N)$ exhibit the same type of convergence behaviour as obtained in \cite{Struwe1985, DT, QingTian, LinWang}  in the non-degenerate case, that is
 subconvergence to a bubble tree of harmonic maps without loss of energy or formation of necks. Conversely, these results do not describe the behaviour of the maps $u_i$ on the parts of the surface where $ 
\inj\to 0$.

As these degenerating regions are contained in the collar neighbourhoods described above we can equivalently consider the behaviour of    
 almost harmonic maps $u_i$ on hyperbolic cylinders $(\Col(\ell_i),g_{\ell_i})$ with  $\ell_i\to 0$.  
 In \cite{HRT} it is shown that such almost harmonic maps subconverge to a full bubble branch in the following sense: 
 
\begin{thm}[Contents of Theorem 1.9 of \cite{HRT}] \label{thm:HRT}
Let $(\Col(\ell_i),g_{\ell_i})$, $\ell_i\to 0$, be a sequence of hyperbolic cylinders as in \eqref{def:collar} and let 
 $u_i:(\Col(\ell_i),g_{\ell_i})\to N$ be a sequence of smooth maps which have bounded energy and for which 
$$\norm{\tau_{g_i}(u_i)}_{L^2(\Col(\ell_i),g_{\ell_i})}\to 0.$$  
Then, after passing to a subsequence, there exists a number $\bar m\in\N_0$ and sequences 
$$-X(\ell_i)=:s_i^0\ll s_i^1\ll\ldots\ll s_i^{\bar m}:=X(\ell_i)$$ 
so that for each $m\in\{1,\ldots, \bar m-1\}$ the shifted maps $u_i(\cdot+s_i^m,\cdot)$ converge to a 
non-trivial bubble branch
locally on $\R\times S^1$ and so that 
$u_i$ maps $[s_i^{m-1}+\La,s_i^m-\La]\times S^1$, $\La$ large, close to a curve in the sense that 
\beq \label{est:oscillation-u-circles-to-zero}
\lim_{\La\to \infty}\limsup_{i\to \infty} \sup_{s \in[s_i^{m-1}+\La, s_i^m-\La]}\osc_{\{s\}\times S^1} u_i =0 \text{ for every } m\in\{1,\ldots,\bar m\}.
\eeq
\end{thm}
Here we recall from \cite{HRT} that we say that maps from cylinders $[-Y_i,Z_i]\times S^1$ with $Y_i,Z_i\to \infty$ converge to a bubble branch if they converge to a harmonic limit $\om_\infty$ weakly in $H^1_{loc}(\R\times S^1)$ and strongly in $H^2_{loc}(\R\times S^1 \setminus S)$ away from a finite (possibly empty) set of points $S\subset \R\times S^1$ where a finite number of bubbles $\om_i$ form. 
We call such a bubble branch non-trivial if the limiting configuration consists of at least one non-trivial harmonic map, or equivalently we call a bubble branch trivial only if no bubbles form and if the map $\om_\infty$ is constant. 
\begin{rmk} \label{rmk:compact}
We also recall that in this convergence of the shifted maps to a bubble branch there can be no loss of energy nor formation of necks on \textit{compact} subsets of $\R\times S^1$. Namely for all suitably large (but fixed) numbers $\La>0$ we know that 
\beq
\label{eq:energy-bb}
E(u_i, [s_i^m-\La,s_i^m+\La]\times S^1)\to E(\om_\infty, [-\La,\La]\times S^1)+\sum E(\om_i)\eeq
as well as that the maps $u_i(\cdot -s_i^m)$ are close in $L^\infty([-\La,\La]\times S^1)$ to maps that are built out of $\om_\infty$ and suitable rescalings of the bubbles, see \cite{HRT} for details. 
\end{rmk}
Conversely,  energy can be lost on the longer and longer cylinders $[s_i^{m-1}+\La,s_i^m-\La]\times S^1$ in the domains that connect the different bubble regions and a description of this loss of energy terms of the Hopf-differential was provided in \cite{Miaomiao} (for harmonic maps) and in \cite{HRT} (for almost harmonic maps).

We note that
while  \eqref{est:oscillation-u-circles-to-zero} ensures that 
the degenerating regions are always mapped close to the curves 
\beq \label{def:hat-u_i} 
\hatumi(s):= \pi_N(\fint_{\{s\}\times S^1} u_i d\th),\quad s\in [s_i^{m-1}+\La,s_i^m-\La], \quad m=1,\ldots,\bar m,
\eeq
$\pi_N$ the nearest point projection from a neighbourhood of $N$ to $N$,  
in the sense that 
\beq
\label{eq:ui-approx-hat-ui} 
 \lim_{\La\to \infty}\limsup_{i\to \infty} \norm{u_i-\hatumi}_{L^\infty([s_i^{m-1}+\La, s_i^{m}-\La]\times S^1)} =0, \quad m=1,\ldots,\bar m,\eeq 
 we cannot expect these curves to collapse to a point in the limit $i\to \infty$ and $\La\to \infty$. 
In particular we cannot expect 
that the asymptotic values
\beq
\label{def:asympt-values}
p_\infty^{m-1,+}:=\lim_{\La\to \infty}\lim_{i\to \infty} u_i(s_i^{m-1}+\La, \cdot) \text{ and } p_\infty^{m,-}:=\lim_{\La\to \infty}\lim_{i\to \infty} u_i(s_i^{m}-\La,\cdot)
\eeq
of neighbouring bubble branches agree. 
 
Instead it is natural to ask whether we can expect the connecting curves  $\hat u_i$  to 'look like' (possibly longer and longer) geodesics. If so this would mean that in the limit $i\to \infty$ we again only obtain objects that are harmonic maps themselves, namely a harmonic map from the limiting surface $\Si$, a finite number of harmonic spheres $\om_i:S^2\to N$ and a number of (possibly infinite length) geodesics, i.e. harmonic maps from suitable intervals. In a situation like that we would then say that the sequence of maps $u_i:M\to N$ converges to a \textit{geodesic bubble tree}.

For maps  $u_i$ which are harmonic, rather than just almost harmonic, the results of Chen-Tian \cite{Chen-Tian-main} and Li-Wang \cite{Li-Wang-main} establish such a convergence to a geodesic bubble tree. Indeed they prove that after passing to a subsequence the connecting curves $\hat u_i^m$ either collapse to a point,  converge (after reparametrising) to a finite length geodesic or that in the limit their image contains an infinite length geodesic. The results of \cite{Chen-Tian-main} furthermore exclude the last possibility in the case of energy minimising maps.
 
Conversely, as explained in \cite{HRT},
we cannot expect to obtain convergence to a \textit{geodesic} bubble tree if we consider the more general case of maps from degenerating hyperbolic surfaces that are almost harmonic in the sense that $\norm{\tau_{g_i}(u_i)}_{L^2(M,g_i)}\to 0$. Indeed, as  pointed out in Proposition 1.14 of \cite{HRT}  we can obtain any $C^2$ curve in $(N,g_N)$ as limit 
of connecting curves for maps with tension $\norm{\tau_{g_i}(u_i)}_{L^2(M,g_i)}\leq C\inj(M,g_i)^{\half}\to 0$ simply by reparametrising such a curve along the collar.

The purpose of this paper is to close the gap left between the results of \cite{Chen-Tian-main, Li-Wang-main} that ensure convergence to a \textit{geodesic} bubble tree for \textit{harmonic} maps and the examples of \cite{HRT} that show that no such result can be true for almost harmonic maps whose tension decays no faster than $\inj(M,g_i)^\half$. 

Indeed, our main result shows that any rate of decay of the tension that is faster than $\inj(M,g_i)^\half$ will force the connecting curves to look like geodesics and hence ensure subconvergence to a geodesic bubble tree.
To prove this it suffices to prove the corresponding result for almost harmonic maps from hyperbolic cylinders with $\ell_i\to 0$.
 
Given such a sequence of maps $u_i: \Col(\ell_i)\to N$ 
with  $\norm{\tau_{g_{\ell_i}}(u_i)}_{L^2(\Col(\ell_i), g_{\ell_i})}=o(\ell_i^\half)$
we can first apply Theorem \ref{thm:HRT} to locate the points $s_i^m$ at which the different bubble branches form.
Instead of splitting the cylinder $[-X(\ell_i),X(\ell_i)]\times S^1$ into
 fixed size bubble regions $[s_i^m-\La,s_i^m+\La]\times S^1$ and the longer and longer cylinders between these sets, we  instead want to consider \textit{extended bubble regions} 
$
B_i^m=[a_i^m,b_i^m]\times S^1$ around $\{s_i^m\}\times S^1$ with $b_i^m-a_i^{m}\to \infty$
and the \textit{connecting cylinders} $[b_i^{m-1},a_i^m]\times S^1$ between these regions, where  $a_i^m$ and $b_i^m$ will be carefully chosen numbers
with
\beq
\label{eq:ai-bi}
s_i^{m-1}\ll b_i^{m-1}\leq a_i^m\ll s_i^m \text{ for } m=1,\ldots,\bar m \text{ while } a_i^0:=-X(\ell_i)\text{ and } b_i^{\bar m}:= X(\ell_i).
\eeq

On the one hand, to ensure that 
there can be no loss of energy or formation of necks on the extended bubble regions we need to know that
\beq \label{claim:no-energy-loss-infty}
\lim_{\La\to \infty} \limsup_{i\to \infty}
E(u_i, [s_i^{m-1}+\La,b_i^{m-1}]\times S^1)+
E(u_i, ([a_i^{m},s_i^m-\La]\times S^1)=0\eeq
and 
\beq 
 \label{claim:no-neck-infty}
 \lim_{\La\to \infty} \limsup_{i\to \infty}  
 \osc_{[s_i^{m-1}+\La,b_i^{m-1}]\times S^1} u_i+
 \osc_{[a_i^{m},s_i^m-\La]\times S^1} u_i=0, \eeq
hold true, compare Remark \ref{rmk:compact}.

On the other hand, we want $a_i^m$ and $b_i^m$ to be so that, 
after reparametrisation by arclength, the restriction of the maps $u_i$ onto the connecting cylinders $[b_i^{m-1},a_i^m]\times S^1$ is essentially described by a (trivial, finite length or infinite length) geodesic.

As we shall see, it is possible to satisfy both of these properties simultaneously if and only if we know that the tension decays strictly faster than $\ell_i^\half$ and indeed we can prove 
\begin{thm}\label{thm:1}
Let 
$u_i:(\Col(\ell_i),g_{\ell_i})\to (N,g_N)$ be a sequence of maps from hyperbolic cylinders with $\ell_i\to 0$ with bounded energy and
 \beq
\label{ass:tension-thm} 
\ell_i^{-\half}
\norm{\tau_{\gi}(u_i)}_{L^2(\Col(\ell_i),g_{\ell_i})}\to 0\eeq
which converges to a full bubble branch as described in Theorem \ref{thm:HRT} above. 

Then we can choose $a_i^m,b_i^m$ 
as in \eqref{eq:ai-bi} so that
 \eqref{claim:no-energy-loss-infty} and \eqref{claim:no-neck-infty} hold, and hence so that there can be no loss of energy or formation of necks on the extended bubble regions $[a_i^m,b_i^m]\times S^1$,
 and so that the images $u_i([b_i^{m-1},a_i^m]\times S^1)$
 of the connecting cylinders subconverge to geodesics in the following sense. 
 
 Let $\hat u_i$ be the connecting curves defined by \eqref{def:hat-u_i} which approximate $u_i$ as described in \eqref{eq:ui-approx-hat-ui} and let 
 $v_i^m:[-c_i^m,c_i^m]\to N$ be the reparametrisation of $\hatumi\vert_{[b_i^{m-1},a_i^{m}]}$ by arclength. Then 
 $$\norm{\tau(v_i^m)}_{L^p([-c_i^m,c_i^m])}\to 0 \text{ for every }p\in [1,2]$$
and hence, after passing to a subsequence, we have 
\begin{enumerate}
\item (Trivial neck case) If $c_i^m\to 0$ then 
 the connecting curves $\hat u_i^m$ collapse to a point so no neck forms between the two bubble branches against which $u_i(\cdot-s_i^{m-1},\cdot)$ and $u_i(\cdot-s_i^m,\cdot)$ converge.
\item (Finite length case) If $c_i^m\to c_m\in (0,\infty)$ then 
the curves $v_i^m(\frac{c_m^i}{c_m}\cdot )$ converge strongly in $W^{2,2}([-c_m,c_m],N)$ to a geodesic which connects 
the points $p_\infty^{m-1,+}$ and  $p_\infty^{m,-}$
in the
 images of the bubble branches against which $u_i(\cdot-s_i^{m-1})$ and $u_i(\cdot-s_i^{m})$ converge. 
\item(Infinite length case)  If $c_i^m\to\infty$ then 
the curves $v_i^m(\cdot +c_i^m)$ and $v_i^m(c_i^m-\cdot)$ converge in $W^{2,2}_{loc}([0,\infty))$ to infinite length geodesics which originate at the points  $p_\infty^{m-1,+}$ and  $p_\infty^{m,-}$. 
\end{enumerate}
\end{thm}
Here 
 $p_\infty^{m-1,+}$ and  $p_\infty^{m,-}$ are given by \eqref{def:asympt-values}.
 
The analogue result also holds true for maps from degenerating tori, albeit with a different sharp rate on the decay of the tension. Namely we show 
\begin{thm}
\label{thm:torus}
Let $(T^2,g_i)$ be a sequence of flat unit area tori whose injectivity radius converges to zero and let $u_i:T^2\to N$ be a sequence of maps whose tension satisfies 
$$\norm{\tau_{g_i}(u_i)}_{L^2(T^2,g_i)}\inj( T^2,g_i)^{-2}\to 0.$$
Then, after passing to a subsequence, the maps $u_i$ converge to a \textit{geodesic} bubble tree as described in Theorem \ref{thm:1}.
\end{thm}
This result is also sharp as we cannot expect the connecting curves to look like geodesics if the rate of decay of the tension is no faster than $\inj(T^2,g_i)^2$, compare Section \ref{sect:tori}.

One of the main difficulties in the proof of the above results is that if we reparametise a curve 
with small velocity 
by arclength then this leads to a sharp increase of the tension.
Roughly speaking for a curve to look like a geodesic we need its tension to be small compared to the square of its velocity. 
A key step towards proving that the connecting curves $\hatumi$ converge to geodesics is hence to obtain a \textit{lower bound} on the velocity with which 
these curves are parametrised. To prove our results we then want to split each interval $[s_i^{m-1}+\La, s_i^{m}-\La]$ 
into a
\begin{itemize}
\item 'relevant' region on which the velocity of $\hatumi$ is large enough that even after reparametrising by arclength the tension still tends to zero.
\item its complement in $[s_i^{m-1}+\La, s_i^{m}-\La]$ for which we want to prove that the velocity of $\hatumi$ so small  that the restriction of $u_i$ to this set cannot make a relevant contribution to the limiting connecting curve nor lead to a loss of energy in the limit $i\to \infty $ and $\La\to \infty$.
\end{itemize}

Importantly, we will be able to prove that the 
 'relevant' region of each such interval is connected. 
This is crucial 
 as we cannot hope to control the tension of the reparametrised curves on sets where $\abs{\hatumi'}$ is too small. If the relevant region was not an interval then the curve $v_i^m$ could change its direction in an uncontrolled way on the irrelevant set, so the best we could hope for would be to obtain a limiting curve that is piecewise geodesic but has corners, rather than a single geodesic arc between the images of neighbouring bubble branches.

Our proofs are based on suitable estimates for almost harmonic maps from degenerating cylinders which we derive in Section \ref{sect:energy}. We then use these estimates in Section \ref{sect:3} to prove our main result Theorem \ref{thm:1} on almost harmonic maps from degenerating hyperbolic surfaces. We conclude this paper with a short Section \ref{sect:tori} where we explain how these arguments can be modified, and indeed simplified, in the case where the domain is a torus.

\section{Estimates for maps from cylinders on low energy regions}\label{sect:energy}
The proof of our main result is based on a 
precise understanding of the behaviour of 
almost harmonic maps away from the regions where bubbles from. We recall from \cite{HRT} that if a sequence of almost harmonic maps converges to a full bubble branch as recalled in Theorem \ref{thm:HRT} then 
for any $\eps>0$ there exist $\La_0$ and $i_0$  so that 
$$E(u_i, \Cyl_1(s))\leq \eps \text{ for all } i\geq i_0 \text{ and } s\in [s_i^{m-1}+\La_0,s_i^m-\La_0],$$
  where here and in the following 
 we write for short $\Cyl_{\lambda}(s):=[s-\lambda,s+\lambda]\times S^1$ and  $\Cyl_{\lambda}= \Cyl_{\lambda}(0)$.

In this section we recall and derive estimates on almost harmonic maps on such low energy regions. 
To state and prove some of these estimates it is more convenient to work with respect to the flat metric $g_E=ds^2+d\theta^2$.
 As the 
energy is conformally invariant and as the euclidean and the hyperbolic tension of maps are related by \eqref{est:relation-tension} we can then easily translate these results to make them applicable to almost harmonic maps from hyperbolic collars. 

So let $X\geq 2$ be any fixed number and let $u\in H^2(\Cyl_X, N)$. We set
\beq 
\label{def:A}
\mathcal{A}(u)=\{s: \abs{s}\geq X-1 \text{ or } E(u,\Cyl_1(s))\geq \eps_0\},\eeq
$\eps_0=\eps_0(N)>0$ determined below, 
and recall that away from $\A(u)\times S^1$ we have the following well known bound on the angular energy 
\beq
\label{def:theta}
\vartheta(s):=\int_{\{s\}\times S^1}\abs{u_\theta}^2 d\theta.
\eeq
\begin{lemma}[Standard angular energy estimates, see e.g Lemma 2.1 \cite{HRT} or Lemma 2.13 \cite{Topping-quantisation}] \label{lemma:ang-energy-standard}
There exists  $\eps_0=\eps_0(N)>0$ 
so that 
for any $u\in H^2(\Cyl_X,N)$ 
and any 
 $s\in [-X,X] \setminus \A(u)$
\beq 
\label{est:ang-energy-standard}
\vartheta(s)\leq C e^{-\dist(s,\A(u))}+C\int_{\Cyl_X} \abs{\tau_{g_E}(u)}^2e^{-\abs{s-q}} dq d\theta.
\eeq 
\end{lemma}
In this section $C$ denotes a constant that only depends on $N$ and an upper bound 
on the energy of $u$.

We note that while the angular energy estimates in \cite{HRT} are only stated for maps with small tension, the $H^2$ estimates stated in Lemma \ref{lemma:H2-standard} below and the trace theorem ensure that the
above estimate is trivially true if the right hand side of \eqref{est:ang-energy-standard} is of order one. 

We also need estimates on the second derivatives of $u$.  Throughout the paper we use the extrinsic viewpoint of considering maps $u$ to $N\hookrightarrow \R^N$ as maps into the surrounding Euclidean space to define higher order derivatives. 
Away from the high energy region we have the following standard $H^2$ estimates.

\begin{lemma}\label{lemma:H2-standard} (see e.g. \cite[Lemma 2.1]{DT} or \cite[Lemma 2.9]{Topping-quantisation})
There exists $\eps_0=\eps_0(N)>0$ so that 
for any $u\in H^2(\Cyl_X,N)$ 
and any 
 $s_0\in [-X,X] \setminus \A(u)$
\beq \label{est:H2-standard}
\int\phisn ^2\abs{\na^2 u}^2 dsd\th\leq C\norm{\phisn \tau_{g_E}(u)}_{L^2(\Cyl_1(s_0),g_E)}^2+ C E(u,\Cyl_1(s_0))
\eeq
and
\beq \label{est:L4-standard}
\int\phisn^2(\abs{u_s}^4+\abs{u_\th}^4) dsd\th \leq CE(u,\Cyl_1(s_0))\cdot \big[ 
\norm{\phisn \tau_{g_E}(u)}_{L^2(\Cyl_1(s_0),g_E)}^2+ E(u,\Cyl_1(s_0))\big],
\eeq
for $\phisn\in C_c^\infty(s_0-1,s_0+1)$ with $\phi\equiv 1$ on $[s_0-\half,s_0+\half]$ and $\abs{\phi'}\leq 4$.
\end{lemma}

In the following we fix $\eps_0=\eps_0(N)>0$ so that both of these lemmas apply and denote by $\A(u) $ the resulting set defined by \eqref{def:A}. 

For integrals that involve angular derivatives of $u$ we obtain the following stronger estimates that only involve the angular energy
$E_\theta(u,\Omega):=\half \int_\Omega \abs{u_\theta}^2 ds d\theta$ instead of the full energy. 

\begin{lemma}\label{lemma:H2-angular-better}
For any $u\in H^2(\Cyl_X, N)$ and any $s_0\in [-X,X]\setminus \A(u)$ we have 
\beq 
\label{est:H2-ang} \int \phisn^2(\abs{u_{\th\th}}^2+\abs{u_{s\th}}^2) dsd\th \leq  C\norm{\phisn \tau_{g_E}(u)}_{L^2(\Cyl_1(s_0),g_E)}^2+ C
E_\theta(u,\Cyl_1(s_0)),
\eeq
$\phisn$ as in Lemma \ref{lemma:H2-standard},
as well as 
\beq \label{est:L4-ang}
\int \phisn ^2\abs{u_\th}^4 dsd\th\leq 
E_\theta(u,\Cyl_1(s_0))\cdot \big[ 
\norm{\phisn \tau_{g_E}(u)}_{L^2(\Cyl_1(s_0),g_E)}^2+ E_\theta(u,\Cyl_1(s_0)) \big].
\eeq
\end{lemma}

\begin{proof}[Proof of Lemma \ref{lemma:H2-angular-better}]
Let  $I_1:=\int\phi^2(\abs{u_{s\th}}^2+\abs{u_{\th\th}}^2) dsd\th$ and 
$I_2 :=\int \phi^2\abs{u_\theta}^4ds d\th$. 
Viewing $I_2$ as the square of the $L^2$ norm of $\phi\abs{u_\th}^2$ and using that $W_0^{1,1}(\Cyl_1)$ embeds continuously into  $L^2(\Cyl_1)$ we get
\beqa \label{est:I2-by-I1} 
I_2 &\leq C\norm{\na(\phi\abs{u_\th}^2)}_{L^1}^2\leq C E_\theta(u,\Cyl_1(s_0))
\cdot \big(I_1+E_\theta(u,\Cyl_1(s_0))\big)
\eeqa
so it suffices to prove the claimed bound \eqref{est:H2-ang} on $I_1$. 
As  \eqref{est:H2-ang} already follows from Lemma \ref{lemma:H2-standard} if  $\norm{\phi\tauE(u)}_{L^2}\geq 1$
 we can furthermore assume that $\norm{\phi\tauE(u)}_{L^2}\leq 1$.

Integration by parts, using also that $u_{ss}+u_{\th\th}=\tauE(u)-A(u)(\na u,\na u)$, yields
\beqas 
I_1& =-\int \partial_s(\phi^2)u_\th u_{s\th} ds d\th +\int\phi^2u_{\th\th}(u_{ss}+u_{\th\th})d\th ds \\
&\leq C I_1^\half E_\theta(u,\Cyl_1(s_0))^\half+I_1^\half \norm{\phi\tauE(u)}_{L^2}
-\int \phi^2u_{\th\th} A(u)(\na u,\na u)ds d\th\\
&\leq \tfrac14 I_1+ C[ \norm{\phi\tauE(u)}_{L^2}^2 +E_\theta(u,\Cyl_1(s_0))
] +
\int\phi^2 u_\th \partial_\th (A(u)(\na u,\na u)) ds d\th.
\eeqas
Using \eqref{est:I2-by-I1} as well as that  $\int \phi^2\abs{\na u}^4 \leq C$, compare \eqref{est:L4-standard}, we can bound 
\beqas 
\int\phi^2 u_\th \partial_\th (A(u)(\na u,\na u)) ds d\th
&\leq C I_1^\half\big(\int \phi^2\abs{\na u}^4 \big)^{\frac14}\cdot I_2^{\frac{1}{4}}+C 
\big(\int \phi^2\abs{\na u}^4 \big)^{\frac12}\cdot I_2^\half\\
&\leq \tfrac14I_1+ C E_\theta(u,\Cyl_1(s_0))
\eeqas 
which, when inserted into the previous estimate, gives \eqref{est:H2-ang}. 
\end{proof}

Combined with Lemma \ref{lemma:ang-energy-standard} we hence obtain that for any $s$ with $\dist(s,\A(u))\geq 2$ 
\beq
\label{est:away-from-A}
\norm{\abs{u_{\theta\theta}}+\abs{u_{s\theta}}+\abs{u_\theta}^2}_{L^2(\Cyl_1(s))}^2 +\vartheta(s)
\leq C R_u(s),
\eeq 
where here and in the following proofs we use the shorthand 
\beq
\label{eq:def-R} 
R_u(s):= \int_{\Cyl_X} \abs{\tau_{g_E}(u)}^2e^{-\abs{s-q}} dq d\theta
+ e^{-\dist(s,\A(u))}.
\eeq
We can use \eqref{est:away-from-A} to prove 
\begin{lemma}\label{lemma:alpha-velocity}
For any $u\in H^2(\Cyl_X,N)$ and any $s_0$ with $\dist(s_0,\A(u))\geq 2$ we have 
\beqa
\label{claim:est-osc-alpha}
\osc_{[s_0-1,s_0+1]}\fint_{S^1} \abs{\partial_s u(\cdot,\th)}d\th& \leq  C\norm{u_{\th\th}}_{L^1(\Cyl_1(s_0),g_E)}+C\norm{\tauE(u)}_{L^1(\Cyl_1(s_0),g_E)}
\leq C R_u(s_0)^{\half}
\eeqa
as well as 
\beq
\label{claim:alpha-beta}
\babs{\fint_{S^1} u_s(s_0,\theta)d\th }\geq  \fint_{S^1} \abs{u_s(s_0,\theta)} d\th -C R_u(s_0)^\half.
\eeq
\end{lemma}
We will apply these estimates later to analyse 
the curves 
$\hat u(s):=\pi_N(\bar u(s))$, where $\bar u(s):= \fint_{\{s\}\times S^1 }u d\th$. Thanks to Lemma \ref{lemma:ang-energy-standard} $\hat u(s)$ is well defined for maps with suitably small tension and for $s$ with suitably large distance from $\A(u)$. 
As $u_s=d\pi_N(u)(u_s)$ we can furthermore estimate
\beqa
\label{est:baru-hatu-der}
\abs{\bar u'-\hat u'}= \abs{\bar u'-d\pi_N(\bar u)(\bar u')}=\babs{\fint_{S^1}  [d\pi_N(u)-d\pi_N(\bar u)] (u_s) d\th}
\leq  C\osc_{S^1} u  \fint_{S^1} \abs{u_s}\leq C \vartheta^\half  \fint_{ S^1} \abs{u_s}.
\eeqa
The lower bound \eqref{claim:alpha-beta} on $\abs{\bar u'}$ from the above lemma hence yield the following lower bound on the velocity of the connecting curve $\hat u$ 

\begin{cor}\label{cor:speed}
For any $u\in H^2(\Cyl_X,N)$ and any 
$s_0$ with 
$\dist(s_0,\A(u))\geq 2$ for which $\hat u$ is well defined in a neighbourhood of $s_0$ we have 
\beq
\abs{\hat u'(s_0)}\geq (1-C\vartheta^\half(s_0)) \fint_{S^1} \abs{u_s(s_0,\theta)} d\th -C R_u(s_0)^\half.
\eeq
\end{cor}

\begin{proof}[Proof of Lemma \ref{lemma:alpha-velocity}]
To establish the first claim of the lemma we can use that
$$\half \partial_s\abs{u_s}^2= u_{ss} \cdot u_s=- u_{\th\th}\cdot u_s+\tauE(u)\cdot  u_s$$
and hence that
 $\abs{\partial_s \abs{u_s}} \leq \abs{u_{\th\th}}+\abs{\tauE(u)}$ almost everywhere. This immediately implies the first estimate of \eqref{claim:est-osc-alpha} while the second estimate of \eqref{claim:est-osc-alpha} follows from \eqref{est:away-from-A} and the definition of $R_u$.

We now want to derive a similar estimate for the oscillation of  $\abs{\bar u'(s)}$  
 which we will later use to prove the second claim of the lemma. To this end we 
%set $\beta(s):=\fint_{\{s\}\times S^1} u_sd\th$ and 
use that 
\beqas
\tfrac{d}{ds}\abs{\bar u'(s)}^2&
=2\bar u'(s)\fint_{\{s\}\times S^1}u_{ss} d\th=2\bar u'(s)\fint_{\{s\}\times S^1} \tauE(u)-A(u)(\na u,\na u) d\th.
\eeqas
As
 $\partial_su\perp A(u)(\na u,\na u)$ we can thus bound 
 \beqas
\abs{\tfrac{d}{ds}\abs{\bar u'(s)}^2}
&\leq 2\abs{\bar u'(s)}\fint_{\{s\}\times S^1}\abs{\tauE(u)} d\th +2\fint_{\{s\}\times S^1} \abs{A(u)(\na u,\na u)} \cdot \abs{u_s-\bar u'} d\theta\\
&\leq 2 \abs{\bar u'(s)}\fint_{\{s\}\times S^1}\abs{\tauE(u)} d\th + 
C (\int_{\{s\}\times S^1} \abs{\na u}^4)^{\half} (\int_{\{s\}\times S^1} \abs{u_{s\theta}}^2)^{\half}.
\eeqas
Integrating this estimate over 
 $I:= [s_0-\half,s_0+\half]$ and then using \eqref{est:away-from-A} and  \eqref{est:L4-standard} gives 
\beqa\label{est:robin-old}
\osc_{I}\abs{\bar u'}^2 &\leq C \sup_{I}\abs{\bar u'} \cdot \norm{\tauE(u)}_{L^2(\Cyl_\half(s_0))}+C \norm{\na u}_{L^4(\Cyl_\half(s_0))}^2  \norm{ u_{\th s}}_{L^2(\Cyl_\half(s_0))}
\\
&\leq
C\sup_{I}\abs{\bar u'}R_u(s_0)^\half+C  E(u,\Cyl_{1}(s_0))^\half R_u(s_0)+C E(u,\Cyl_{1}(s_0)) R_u(s_0)^\half\\
&\leq \sup_{I}\abs{\bar u'}R_u(s_0)^\half+
C  R_u(s_0)+C E(u,\Cyl_{1}(s_0))^\half R_u(s_0)^\half.
\eeqa
As \eqref{est:away-from-A} furthermore allows us to bound 
\beqas
E(u,\Cyl_1(s_0))&
\leq C\int_{s_0-1}^{s_0+1}\int_{S^1} \abs{u_s(s,\th)-\bar u'(s)}^2d\theta  +\abs{\bar u'(s)}^2+\vartheta(s)ds  \\
&\leq C\norm{u_{s\th}}_{L^2(\Cyl_1(s_0))}^2 +C\sup_{[s_0-1,s_0+1]}\abs{\bar u'}^2+CR_u(s_0)\leq C \sup_{[s_0-1,s_0+1]}\abs{\bar u'}^2 + CR_u(s_0),
%\\
\eeqas
we hence know that 
%\beqa
$\osc_{I}\abs{\bar u'}^2 \leq  C  R_u(s_0)^{\half} \sup_{I}\abs{\bar u'} +C R_u(s_0).$
%\eeqa

We now set $\al(s):= \fint_{\{s\}\times S^1} \abs{u_s} d\theta$ and use that $\abs{\bar u'}\leq \al$ and that we have already 
established 
the bound \eqref{claim:est-osc-alpha} on the oscillation of $\al$. This allows us to conclude that 
\beqas
\osc_{I}\abs{\bar u'}^2 &\leq  C  R_u(s_0)^{\half} \sup_{I}\al +C R_u(s_0)\leq C\al(s_0) R_u(s_0)^{\half} + C R_u(s_0).
\eeqas
As $\int_I\al-\abs{\bar u'}\leq \int_{I\times S^1}\abs{u_s-\fint_{S^1} u_s}\leq \norm{u_{s\th}}_{L^2(\Cyl_\half(s_0))}\leq CR_u(s_0)^\half$ we furthermore have 
$$\int_I \al^2-\abs{\bar u'}^2 \leq 2\sup_I\al\cdot \int_I \al-\abs{\bar u'}\leq C(\al(s_0)+\osc_I\al)\cdot R_u(s_0)^\half\leq C\al(s_0) R_u(s_0)^\half+CR_u(s_0).$$
Combined, these two estimates give
\beqas
\abs{\bar u'(s_0)}^2
&\geq \int_I \al^2 -\int_I \al^2-\abs{\bar u'}^2ds -\osc_I\abs{\bar u'}^2
 &\geq \al(s_0)^2-C[\al(s_0)R_u(s_0)^\half +R_u(s_0)].
\eeqas

If $\al(s_0)\geq C_1 R_u(s_0)^\half$ for a sufficiently large but fixed $C_1\geq 1$, we hence deduce that 
$$\abs{\bar u'(s_0)}^2\geq \al(s_0)(\al(s_0)-C R_u(s_0)^\half)\geq \abs{\bar u'(s_0)}(\al(s_0)-CR_u(s_0)^\half)$$
since $\abs{\bar u'}\leq \al$ and since the term in the bracket is positive. We hence obtain the claimed bound
 \eqref{claim:alpha-beta} in this case which suffices to complete the proof of the lemma as \eqref{claim:alpha-beta} is trivially true (for $C\geq C_1$) if $\al(s_0)< C_1R_u(s_0)^\half$. 
\end{proof}
To translate the estimates derived above to the hyperbolic setting we recall that the tension with respect to the hyperbolic metric $g=\rho^2g_E$ is given by 
\beq
\label{eq:relation-tension-pw}
\tau_{g}(u)= \rho_\ell^{-2}\tauE(u)
\eeq
which in particular implies that for any $I\subset [-X(\ell),X(\ell)]$
\beq
\label{est:relation-tension}
\norm{\tauE(u)}_{L^2(I\times S^1,g_E)}\leq \sup_{I} \rho \cdot \norm{\tau_g(u)}_{L^2(I\times S^1, g)}.
\eeq
We also recall that the conformal factor $\rho=\rho_\ell$ is bounded uniformly on collars $\Col(\ell)$ for $\ell$ in a bounded range, say $\ell\in (0,\arsinh(1))$, and  that 
for every $1\leq \La\leq X(\ell)$
\beq
\label{est:rho-ends-of-collar}
C^{-1} \La^{-1}\leq \rho(X(\ell)-\La)\leq C \La^{-1}.
\eeq 
As $\abs{\log(\rho)'}= \rho \abs{\sin(\frac{\ell}{2\pi}s)}$, and thus in particular 
$\abs{\log(\rho)'}\leq \frac14$ 
 away from the ends of the collar, we can bound  
\beq \label{est:rho-al-exp}
\rho^2(s')e^{-\half\abs{s-s'}}\leq C \rho^2(s), \quad \text{ for all }s,s'\in [-X(\ell),X(\ell)]
\eeq
and hence get that 
the quantity 
$R_u$ appearing in the above lemmas is bounded by 
  \beq
  \label{est:F-by-F-hyp}
  R_u(s)\leq C\rho^2(s)  
  \int_{\Col(\ell)}\abs{\tau_{g_\ell}(u)(q,\theta)}^2 e^{-\half\abs{s-q}} dv_{g_\ell}(q,\theta)+e^{-\dist(s,\A(u))}.
  \eeq
In particular, for every interval 
 $I\subset [-X(\ell),X(\ell)]\setminus \A(u)$ 
  \beq
  \label{est:F-by-tension}
 \sup_I R_u+ 
  \int_I R_u\leq C \sup_I \rho^{2}\cdot \norm{\tau_g(u)}_{L^2}^{2}+2 e^{-\dist(I,\A(u))}
  \eeq
 where here and in the following the norm of $\tau_g(u)$ is computed over $(\Col(\ell),g)$ unless specified otherwise. 
 We will furthermore use that 
 \beq 
\label{est:F-to-half} 
 \int_I (R_u)^\half  \leq (\int_I \rho^2 ds)^\half \norm{\tau_g(u)}_{L^2}+4 e^{-\half\dist(I,\A(u))}\leq C\sup_I \rho^\half\cdot  \norm{\tau_g(u)}_{L^2}+4 e^{-\half\dist(I,\A(u))}.
  \eeq
  Here and in the following we still define the high energy set $\A(u)$ by the formula \eqref{def:A} that involves the energy of maps $u$ on cylinders of unit length with respect to $g_E$ in the corresponding collar coordinates. Similarly we continue to denote by 
   $\dist(\cdot,\A(u))$  the (euclidean) distance of numbers $s\in [-X(\ell),X(\ell)]$ respectively intervals $I\subset [-X(\ell),X(\ell)]$ from the high energy set $\A(u)$. 
   
We first combine these estimate with the well known bounds on the angular energy that we recalled in Lemma \ref{lemma:ang-energy-standard} to show
\begin{lemma}\label{lemma:transition}
Let $u:\Col(\ell)\to N$ be a $H^2$-map from a hyperbolic cylinder $(\Col(\ell), g)$ with $\ell\in (0,\arsinh(1))$. 
Then for every $s\in [-X(\ell),X(\ell)]$ we can bound 
\beqa
\label{est:us-low-energy}
\int_{\{s\}\times S^1} \abs{u_s}^2 -\abs{u_\theta}^2 d\theta
&\leq C\ell +C\rho(s) \norm{\tau_g(u)}_{L^2}
\eeqa
while for every interval  $I\subset [-X(\ell),X(\ell)]\setminus \A(u)$
\beq\label{claim:lemma-energy-transition}
E(u,I\times S^1) \leq C e^{-\dist(I,\A(u))}+ 
C\sup_I\rho^2 \cdot \norm{\tau_g(u)}_{L^2}^2
+C\abs{I} \cdot(\ell+\sup_{I}\rho \cdot \norm{\tau_g(u)}_{L^2})
\eeq
and \beqa 
\label{claim:lemma-I_j-1} 
\osc_{I\times S^1} u \leq &  C   e^{-\half\dist(I,\A(u))}+ C
\sup_I\rho^\half \norm{\tau_g(u)}_{L^2}+ C\abs{I} \cdot(\ell+\sup_{I}\rho \norm{\tau_g(u)}_{L^2})^\half.
\eeqa
\end{lemma}
\begin{proof}[Proof of Lemma \ref{lemma:transition}]
Given $u$ as in the lemma we let  $\psi(s)=\abs{u_s}^2-\abs{u_\th}^2$ be the 
 real part of the function that represents the Hopf-differential $\Phi(u)=(\abs{u_s}^2-\abs{u_\theta}^2-2 i u_su_\theta)(ds+ i d\theta)^2$ in collar coordinates and 
 set  $I_\psi(s):= \int_{\{s\}\times S^1} \psi$. It is well known that the antiholomophic derivative of the Hopf-differential is bounded in terms of the tension and we exploit this to write
\beqs
\partial_s I_{\psi}= \int_{S^1} \partial_s\psi=
\int_{S^1}\partial_s \psi+\partial_\theta (2u_s u_\theta) =2\int_{ S^1}  (u_{ss}+u_{\th\th}) u_s= 2\int_{S^1} \tauE(u) u_s.
\eeqs
Combined with \eqref{est:relation-tension} 
this implies that for any interval $J\subset [-X(\ell),X(\ell)]$ 
\beq
\label{est:osc-phi} \osc_{J} I_{\psi}\leq C E(u)^\half\cdot
\norm{\tauE(u)}_{L^2(J\times S^1,g_E)}\leq C \sup_J{\rho}\cdot \norm{\tau_g(u)}_{L^2}.\eeq
We apply this estimate for 
$J=[s-\half X(\ell),s]$ if $s\geq 0$ respectively for $J=[s,s+ \half X(\ell)]$ if $s<0$ as this ensures that $\sup_J\rho\leq \sqrt{2} \rho(s)$. Indeed, if 
$\abs{s}\geq \frac{X(\ell)}{2}$ then $\sup_J \rho=\rho(s)$ while
otherwise $\sup_J \rho\leq \rho(\frac{X(\ell)}{2})\leq \sqrt2 \rho(0)\leq \sqrt2\rho(s)$. 
As 
$\abs{J}=\frac12X(\ell)\geq c \ell^{-1}$ for a universal $c>0$, we can bound $ \inf_J I_\psi\leq C\ell E(u)\leq C\ell$ 
and thus obtain from \eqref{est:osc-phi} that 
\beq
I_\psi(s)\leq \inf_J I_\psi+\osc_J I_\psi\leq 
C\ell +C\rho(s) \norm{\tau_g(u)}_{L^2} \text{ for every } s\in [-X(\ell),X(\ell)],
\eeq
as claimed in \eqref{est:us-low-energy}. 

The second claim 
\eqref{claim:lemma-energy-transition} of the lemma  immediately follows by integrating \eqref{est:us-low-energy}  over the given interval $I\subset [-X(\ell),X(\ell)]\setminus \A(u)$ as
Lemma \ref{lemma:ang-energy-standard} and \eqref{est:F-by-tension} imply that
\beq \label{est:int-theta-I}
\sup_I \vartheta +
\int_I \vartheta \leq C\sup_I \rho^2\norm{\tau_g(u)}_{L^2}^2+Ce^{-\dist(I,\A(u))}.
\eeq
Similarly, from \eqref{est:us-low-energy} we obtain
 \beqas %\label{est:proof-osc-Ij-1}
\osc_{I\times S^1} u&\leq 
2\sup_{s\in I} \osc_{\{s\}\times S^1} u+\frac{1}{2\pi} \int_{I\times S^1} \abs{u_s} dsd\theta\ \\
&\leq C\sup_{ I} \vartheta^{\half}+C \int_I \vartheta^\half ds 
+ C \abs{I} \cdot \sup_I \big( \int_{\{s\}\times S^1} \abs{u_s}^2-\abs{u_\theta}^2\big)^{\half}
\\
& \leq C\sup_{ I} \vartheta^{\half} +C\int_I \vartheta^{\half}+C  \abs{I} \cdot(\ell+\sup_{I}\rho \cdot \norm{\tau_g(u)}_{L^2})^\half.
 \eeqas 
As Lemma \ref{lemma:ang-energy-standard} and \eqref{est:F-to-half} imply that
 \beq
 \label{est:int-theta-half}
 \int_I \vartheta^\half  \leq C\sup_I \rho^\half\cdot \norm{\tau_g(u)}_{L^2}+Ce^{-\half\dist(I,\A(u))} 
\eeq
we hence obtain the final claim \eqref{claim:lemma-I_j-1} of the lemma.
\end{proof}

We shall later apply Lemma \ref{lemma:transition} to prove that there can be no loss of energy or formation of necks on regions of the collar with distance of order $O(\abs{\log(\ell_i)})$ from the high energy region $\A(u_i)$. Conversely, for points which have larger distance from the high energy set, and for which Lemma \ref{lemma:ang-energy-standard} hence gives a stronger bound on the angular energy, we will use the following lemma.
\begin{lemma}
\label{lemma:osc-al-log}
Let $u:\Col(\ell)\to N$ be a map from a hyperbolic cylinder $(\Col(\ell), g)$ with $\ell\in (0,\arsinh(1))$
 and let 
$I\subset [-X(\ell),X(\ell)]$ be any interval with $\dist(I,\A(u))\geq 4\abs{\log(\ell)}$. Then 
$\al(s):= \fint_{\{s\}\times S^1} \abs{u_s} d\th$ satisfies 
 \beq
 \label{est:al-difference}
 \abs{\al(s)-\al(\tilde s)}\leq C\ell^2+C\max(\rho(s),\rho(\tilde s))^\half\norm{\tau_g(u)}_{L^2(\Col(\ell),g_\ell)} \text{ for all } s,\tilde s\in I
 \eeq
 and, denoting by  $s_0$ the element of $\bar I$ with 
 $\rho(s_0)=\inf_{I} \rho$, we 
  furthermore obtain that 
\beq \label{est:al-integral}
\int_I \al\leq C \rho(s_0)^{-1}\al(s_0) +C\ell+C\rho(s_0)^{-\half}
\norm{\tau_g(u)}_{L^2(\Col(\ell),g_\ell)}.
\eeq
\end{lemma}
\begin{proof}
Since $e^{-\half \dist(I,\A(u))}\leq C\ell^2$ 
we can 
integrate \eqref{claim:est-osc-alpha} from $s$ to $\tilde s$ and 
 use
\eqref{est:F-to-half} to obtain the first claim of the lemma as 
the maximum of $\rho$ on any closed interval is achieved at one of the endpoints.  

To obtain the second claim we then use that
\eqref{est:al-difference} gives
\beq
\label{est:1111} 
\al(s)\leq \al(s_0)+C\ell^2+C (2^{j}\rho(s_0))^{\half} \norm{\tau_g(u)}_{L^2} \text{ for } s\in I \text{ with } 2^{j-1}\rho(s_0)\leq \rho(s)\leq 2^j \rho(s_0).\eeq
At the same time we can use that
\beq
\label{est:2222}
\mathcal{L}^1(\{s\in[-X(\ell),X(\ell)]: \half \eta\leq \rho(s)\leq \eta\})\leq C \eta^{-1}  \text{ for any } \eta>0.\eeq
Indeed, for $\half \eta\leq \sqrt{2}\rho(0)=\frac{\ell}{\sqrt{2}\pi}$ this follows as $X(\ell)\leq C \ell^{-1}$ while for larger values of $\eta$ this set consists of two intervals
on which $\cos(\frac{\ell}{2\pi}s)\leq \frac{1}{\sqrt{2}}$ and thus 
 $\abs{\log(\rho)'}= \rho \abs{\sin(\frac{\ell}{2\pi}s)}\geq \frac{1}{\sqrt{2}} \rho\geq (2\sqrt{2})^{-1} \eta$. Hence these intervals have length no more than $2\sqrt{2}\log 2\eta^{-1}$. 

From \eqref{est:1111} and \eqref{est:2222}, which also implies that $\abs{I}\leq C \rho(s_0)^{-1}\leq C \ell^{-1}$, we thus obtain the claimed estimate of
\beqas 
\int_I \al(s) ds &\leq  (\al(s_0)+C\ell^2)\abs{I} +C  \norm{\tau_g(u)}_{L^2}\sum_{j\geq 0} (2^j \rho(s_0))^{\half} \cdot (2^j \rho(s_0))^{-1}\\
&\leq C\rho(s_0)^{-1}\al(s_0)+ C\ell+ 
C \rho(s_0)^{-\half}  \norm{\tau_g(u)}_{L^2(\Col(\ell))}.
\eeqas
\end{proof}
\section{Proof of Theorem \ref{thm:1}}
\label{sect:3}
We now turn to the proof of our main results in the hyperbolic setting. So let $u_i$ be a sequence of almost harmonic maps from hyperbolic cylinders 
$(\Col(\ell_i),g_i=g_{\ell_i})$,  $\ell_i\to 0$, with 
\beq 
\label{def:eps-1}
\eps_i:= \norm{\tau_{g_i}(u_i)}_{L^2(\Col(\ell_i),g_{i})}\ell_i^{-\half}\to 0
\eeq 
which converges to a full bubble branch as recalled in the introduction, see  \cite{HRT} for more detail.
The choice of the numbers  $\{s_i^m\}_{m=0}^{\bar m}$ 
 in \cite{HRT} ensures that 
there exists a number $\La_0$ (allowed to depend on the specific sequence $u_i)$ for which 
$\A(u_i)\subset \bigcup_m [s_i^m-\La_0,s_i^m+\La_0] $
for all $i$. As the bubble branches against which the shifted maps 
$u_i(\cdot-s_i^m)$, $1\leq m\leq \bar m-1$ converge are non-trivial and as we always have $\pm X(\ell_i)\in \A(u_i)$ we furthermore know that 
$ \A(u_i)\cap  [s_i^m-\La_0,s_i^m+\La_0]\neq \emptyset $ for every $m=0,\ldots,\bar m$ and sufficiently large $i$. 

In the following we can thus use that 
\beq
\label{est:equiv-dist}
\dist(s,\A(u_i))-\La_0\leq \min_m \abs{s-s_i^m}\leq \dist(s,\A(u_i))+\La_0
\eeq
for every $s\in [-X(\ell_i),X(\ell_i)]$, and hence in particular that factors like $e^{-\dist(s,\A(u))}$ are bounded by  $C e^{-\min_m \abs{s-s_i^m}}$. Here and in the following $C$ denotes a constant that is allowed to depend on the specific sequence of maps $u_i$, but is of course independent of $i$. 

We first need to decide which parts of $[s_i^{m-1}+\La_0, s_i^m-\La_0]\times S^1$
 we want to include in the extended bubble regions around $s_i^{m-1}$ respectively $s_i^m$ and which (central) part of this cylinder we want to view instead as the connecting cylinder $[b_i^{m-1},a_i^m]\times S^1$ between the 
extended bubble regions.

To begin with we look at the auxiliary sets 
\beq \label{def:Iim}
I_i^m:=\{s\in[s_i^{m-1},s_i^m]: \dist(s,\A(u_i))\geq 4 \abs{\log(\ell_i)}+1\},
\eeq
where here and in the following we can fix
$m\in \{1,\ldots, \bar m\}$ and only need to consider sufficiently large indices $i$. We can hence in particular assume that   $4\abs{\log(\ell_i)}+1\geq \La_0+2$ which ensures that these sets are closed (possibly empty) intervals. 

If $I_i^m$ is not empty we furthermore let  $t_{i}^m$ be the 
element of $I_i^m$ for which the conformal factor is minimal, i.e. set $t_i^m=0$ if $0\in I_i^{m}$ while $t_i^m=\pm \min_{I_i^m}\abs{s}$ if $I_i^m\subset \R^\pm $. 

If, after passing to a subsequence, either  $I_i^m=\emptyset$ for all $i$ or  $I_i^m\neq \emptyset$ but 
\beq
\label{ass:to-zero}
\rho_i(t_i^m)^{-1}\al_i(t_i^m)\to 0 \text{ as } i\to \infty\eeq
for $\al_i(s):= \fint_{\{s\}\times S^1}\abs{\partial_s u_i}$ 
then we set 
$
b_i^{m-1}=a_i^{m}= \half (s_i^{m-1}+s_i^m)$. In this case we hence end up with a trivial set as connecting cylinder and will prove that no neck forms between the bubble branches against which $u_i(\cdot -s_i^{m-1})$ respectively $u_i(\cdot -s_i^m)$ converge. 

Conversely if, after passing to a subsequence,
\beq
\label{ass:not-to-zero}
\rho_i(t_i^m)^{-1}\al_i(t_i^m)\geq c_0 \text{ for all } i \text{ and  some } c_0>0, 
\eeq 
then we choose 
$b_i^{m-1}$ and $ a_i^{m}$ as the minimal and maximal element of $I_i^m$ for which 
\beq\label{def:ai-bi} \rho_i(s)^{-1}\al_i(s)\geq \de_i:=\max(\eps_i,\ell_i)^{\half}.\eeq
As $\de_i\leq c_0$ (for sufficiently large $i$) these numbers are well defined and satisfy $b_i^{m-1}\leq t_i^m\leq a_i^{m}$.

We note that in both cases \eqref{eq:ai-bi} is satisfied since $s_i^m-s_i^{m-1}\to \infty$ and since 
\eqref{est:equiv-dist} ensures that $\dist(I_i^m,\{s_i^{m-1},s_i^{m}\})\geq 4 \abs{\log(\ell_i)}-\La_0\to \infty$ if $I_i^m\neq \emptyset$. 

As a first step towards the proof of our main theorem we now want to establish that this choice of $a_i^m$ and $b_i^m$ ensures
 that 
\eqref{claim:no-energy-loss-infty} and \eqref{claim:no-neck-infty} hold true, and hence that there is no loss of energy nor formations of necks on the extended bubble regions 
$B_{i}^m =[a_i^m,b_i^m]\times S^1$. 

By symmetry it suffices to analyse the behaviour of $u_i$ for $s\in [a_i^m, s_i^m-\La]$.
To this end we first consider the set 
$[a_i^m, s_i^m-\La]\setminus I_i^m$ which is an interval of length 
no more than $O(\abs{\log\ell_i})$ whose distance from $\A(u)$ is at least $\La-\La_0$. Using additionally that the conformal factors $\rho_i$ are uniformly bounded we hence obtain from Lemma \ref{lemma:transition} and our main assumption \eqref{def:eps-1} that  
\beqs
E(u_i, ([a_i^m, s_i^m-\La]\setminus I_i^m) \times S^1)
\leq C e^{-\La}+C\ell_i\eps_i^2+ C\abs{\log\ell_i}(\ell_i+\eps_i\ell_i^{\half})
\eeqs
and 
\beqs
\osc_{([a_i^m, s_i^m-\La]\setminus I_i^m) \times S^1} u_i \leq Ce^{-\La/2}+C\ell_i^\half\eps_i+C\abs{\log\ell_i}(\ell_i^\half+\eps_i^\half\ell_i^{\frac14})
\eeqs
both tend to zero 
as $i\to \infty$ and $\La\to \infty$.

It thus remains to bound the energy and oscillation of $u_i$ on $(I_i^m\cap\{s\geq a_i^m\})\times S^1$ and to this end we want to prove that
\beq\label{est:int-al-Iim}  
\int_{I_i^m\cap\{s\geq a_i^m\}}\al_i(s) ds\to 0 \text{ as } i\to \infty.
\eeq
This is trivially true if $I_i^m=\emptyset$ and is also true if \eqref{ass:to-zero} holds 
as in this situation Lemma \ref{lemma:osc-al-log} 
implies that 
\beqs
\int_{I_i^m}\al_i(s) ds\leq C\rho(t_i^m)^{-1}\al_i(t_i^m)+C\ell_i+C\rho(t_i^m)^{-\half}\ell_i^\half \eps_i\to 0 \text{ as } i\to \infty.
\eeqs

In the remaining case where \eqref{ass:not-to-zero} holds we can instead use that $a_i^m$ is chosen so that 
$\rho_i^{-1}(s)\al_i(s)\leq \de_i\to 0$ for all $s\in I_i^m$ with $s\geq a_i^m$. We can hence apply the analogue argument (with $t_i^m$ replaced by the element $\tilde t_i^m$ of $I_i^m\cap\{s\geq a_i^m\}$ which minimises $\rho$) 
to conclude that also in this case \eqref{est:int-al-Iim} holds.

Having established \eqref{est:int-al-Iim}, and hence that the oscillation of 
$ \fint u_i(\cdot,\theta)d\theta$ over $I_i^m\cap\{s\geq a_i^m\}$ tends to zero, we immediately deduce that 
$$\osc_{(I_i^m\cap \{s\geq a_i^m\})\times S^1} u_i\to 0$$
as the standard angular energy estimates recalled in 
Lemma \ref{lemma:ang-energy-standard} ensure that 
 the oscillation of $u_i$ over circles $\{s\} \times S^1$, $s\in I_i^m$, is bounded by 
$C\vartheta_i(s)^\half \leq C\ell_i +C \ell_i^\half\eps_i\to 0$. 

Combining 
\eqref{est:int-al-Iim}  with  Lemmas \ref{lemma:ang-energy-standard} and \ref{lemma:H2-angular-better} furthermore allows us to deduce that   
\beqas
E((I_i^m\cap\{s\geq a_i^m\})\times S^1) &\leq C\int_{I_i^m\cap\{s\geq a_i^m\}} \al_i^2+ \vartheta_i +\int_{S^1}
\abs{u_{s\th}}^2d\th \, ds\\
&\leq C \int_{I_i^m\cap\{s\geq a_i^m\}} \al_i ds +C\ell_i^2+ C\ell_i\eps_i^2\to 0,
\eeqas 
since Lemma \ref{lemma:H2-standard} and the trace theorem ensure that the $\al_i$ are uniformly bounded away from $\A(u_i)$. 

This completes the analysis of the maps $u_i$ on the extended bubble regions and it remains to study the behaviour of the curves $\hat u_i$, and their reparametrisations $v_i^m$, on the intervals $(b_i^{m-1},a_i^m)$. 

We recall that these intervals are  non-empty only if  \eqref{ass:not-to-zero} holds and note that by symmetry 
 we can assume without loss of generality that $\abs{a_i^{m}}\geq \abs{ b_i^{m-1}}$ and hence that 
$\rho\leq \rho(a_i^{m})$ on $[b_i^{m-1},a_i^m]$. 

As the choice of $a_i^m$ ensures that $\al_i(a_i^m)\geq \rho_i(a_i^m) \de_i$ and as $\ell_i\leq 2\pi\rho_i$ we can first apply the estimate 
\eqref{est:al-difference} from Lemma \ref{lemma:osc-al-log} to see that for every $s\in [b_i^{m-1},a_i^{m}]$
\beqa\label{est:lower-al-in-a-proof}
\al_i(s)&\geq \al_i(a_{i}^{m})-C\ell_i^2-C
\rho_i(a_{i}^{m})^{\half} \norm{\tau_{g_i}(u_i)}_{L^2}\\
&\geq \rho_i(a_{i}^{m})\big[ \de_i -C \ell_i-C \eps_i\big]=(1-o(1))\rho_i(a_{i}^{m}) \de_i\geq (1-o(1))\rho_i(s)\de_i
\eeqa
where we use in the penultimate step that $\de_i=(\max(\eps_i,\ell_i))^\half \gg \max(\eps_i,\ell_i)$.

As $R_{u_i}(s)+\vartheta_i(s)\leq C R_{u_i}(s)\leq C\rho_i^2(s) \ell_i\eps_i^2+C\ell_i^4$, compare \eqref{est:F-by-F-hyp} and Lemma \ref{lemma:ang-energy-standard}, we can use Corollary \ref{cor:speed} to see that
\beqa \label{est:al-by-velocity}
\abs{\hat u_i'}&\geq (1-\vartheta_i^\half)\al_i-C\ell_i^2-C\rho_i\ell_i^\half \eps_i\geq 
(1-o(1)-C\ell_i\de_i^{-1}-C\ell_i^\half \de_i^{-1}\eps_i)\al_i\geq (1-o(1)) \al_i
\eeqa
on $ [b_i^{m-1},a_i^m]$. 
In particular, for all sufficiently large $i$, we have
\beq 
\label{claim:lower-bound-velocity}
\abs{\hat u_i'(s)}\geq \half \al_i\geq \half \de_i \rho_i(s) \text{ for  } s\in [b_i^{m-1},a_i^{m}].
\eeq
Next we prove that the (euclidean) tension of the curve $\hat u_i$ is bounded by 
 \beqa \label{est:tension-ui}
\abs{\tauE(\hatumi)} 
  &\leq C\int_{S^1} \abs{\tauE(u_i)} d \th+
  C\vartheta_i^\half \bigg[ 
\abs{\hatumi '}^2+ \abs{\hatumi '} \int_{S^1} \abs{\partial_{s\th} u_i}d\th+\vartheta_i  
\bigg].
\eeqa
To see this we first use that $\hat u_i'=d\pi_N(\bar u_i)(\bar u'_i)$ for $\bar u_i(s)=\fint_{\{s\}\times S^1} u_i $
and that  $-d^2\pi(p)\vert_{T_pN\times T_pN}=A(p)$, $p\in N$, to write 
 \beqa
 \label{est:hallo}
\tauE(\hatumi)&=\hatui''+A(\hatui)(\hatui',\hatui')
=d\pi_N(\bar u_i)(\bar u_i'')+d^2\pi_N(\bar u_i)(\bar u_i',\bar u_i')-d^2\pi_N(\hatui)(\hatui',\hatui').
\eeqa
As
$\abs{\bar u_i-\hat u_i}\leq \osc_{S^1} u_i\leq C\vartheta_i^\half$ and as $\abs{\bar u_i'-\hat u_i'}\leq  C \vartheta_i^\half \al_i\leq 
 C \vartheta_i^\half \abs{\hat u_i'}$, compare \eqref{est:baru-hatu-der} and \eqref{claim:lower-bound-velocity},
we can bound 
$$\abs{d^2\pi_N(\bar u_i)(\bar u_i',\bar u_i')-d^2\pi_N(\hatui)(\hatui',\hatui')}\leq C\vartheta_i^\half \abs{\hat u_i'}^2,$$
so it remains to estimate the first term in \eqref{est:hallo}. 

As $\bar u_i''(s)=\fint_{S^1}\partial_{ss} u_i=\fint_{S^1}\Delta_{g_E} u_i=\fint_{S^1}\tauE(u_i)+A(u_i)(\na u_i,\na u_i)$ 
and as $d\pi_N(u_i)(A(u_i)(\na u_i,\na u_i))=0$  we have
\beqas
\abs{d\pi_N(\bar u_i)(\bar u_i'')}&\leq C\int_{S^1} \abs{\tauE(u_i)}+\norm{u_i-\bar u_i}_{L^\infty(S^1)} \int_{S^1}\abs{\na u_i}^2\leq C\int_{ S^1} \abs{\tauE(u_i)}+C\vartheta_i^\half \int_{ S^1}\abs{\na u_i}^2.
\eeqas
This yields the claimed estimate \eqref{est:tension-ui} as  $\al_i=\fint\abs{\partial_s u_i}\leq 2\abs{\hat u_i'}$ for all sufficiently large $i$ and as we can hence  bound
\beqas
\int_{S^1}\abs{\na u}^2 &\leq \vartheta_i+ C\al_i \norm{\partial_s u_i}_{L^\infty(S^1)} \leq\vartheta_i+ C \al_i( \al_i+\int_{S^1} \abs{\partial_{s\th} u_i})\leq \vartheta_i+ C \abs{\hat u_i'} ( \abs{\hat u_i'}+\int_{S^1} \abs{\partial_{s\th} u_i}).\eeqas
We now reparametrise $\hatumi\vert_{[b_i^{m-1},a_i^m]}$ by arclength and estimate the tension of the resulting curve $v_i^m=\hatumi\circ s_i^m$, which is well defined as \eqref{claim:lower-bound-velocity} ensures that $\hat u_i'\neq 0$ on this interval. To lighten the notation we write for short $v_i=v_i^m$ and $s_i=s_i^m$. 

So let $c_i^m=\half\int_{b_i^{m-1}}^{a_i^{m}}\abs{\hatumi '} ds$ and let 
$s_i:[-c_i^m,c_i^m]\to [b_i^{m-1}, a_i^{m}]
$ be the increasing bijection with $\abs{\dot s_i}^2 \cdot \abs{\hatumi '\circ s_i}^2=1$.
Differentiating this relation and using that 
 $\hatui ''\cdot \hatui'=\tauE(\hatui) \hatui'$ 
yields 
$$\abs{\ddot{s_i}}\leq \abs{\dot{s_i}}^{2}\abs{\hatumi '\circ s_i}^{-1} \abs{\tauE(\hatumi)\circ s_i }=
\abs{\hatumi '\circ s_i}^{-3} \abs{\tauE(\hatumi)\circ s_i }.$$ 
As $\tauE(v_i)= \abs{\dot s_i}^2 \tauE(\hatumi)\circ s_i + \ddot{s_i} \hatumi '\circ s_i $ we can thus use 
\eqref{est:tension-ui}  to bound 
\beqas %\label{est:tension-vi}
\abs{\tauE(v_i)\circ s_i^{-1}}&
\leq 2\abs{\hatumi '}^{-2} \abs{\tauE(\hatumi)}\\
&\leq C \abs{\hatumi '}^{-2}
\int_{S^1} \abs{\tauE(u_i)} d \th  +C \vartheta_i^\half \bigg(1+\abs{\hatumi '}^{-2} \vartheta_i+\abs{\hatumi '}^{-1} \int_{S^1}\abs{\partial_{s\th} u_i}d\th\bigg).
\eeqas
For any $p\in [1,2]$ we can hence estimate
\beqas %\label{est:tension-rescaled-proof-1}
\norm{\tauE(v_i)}_{L^p([-c_i^m,c_i^m])}^p&=
\int \abs{\hatumi '}\abs{\tauE(v_i)\circ s_i^{-1}}^p ds\\
&\leq C\int 
\abs{\hatumi '}^{-2p+1}
\abs{\tauE(u_i)}^p +\vartheta_i^{\frac{p}2}\abs{\hatumi '}  +
\abs{\hatumi '}^{-2p+1} \vartheta_i^\frac{3p}2+\abs{\hatumi '}^{-p+1} \abs{\partial_{s\th} u_i}^p ds d\theta
\eeqas
where we integrate over $s\in [b_i^{m-1},a_i^{m}]$ and $\theta\in S^1$. Combined with the lower bound \eqref{claim:lower-bound-velocity} on the velocity of $\hat u_i$ and the fact that $\rho_i\geq c\ell_i$ we hence get 
\beqas %\label{est:tension-rescaled-proof-1}
\norm{\tauE(v_i)}_{L^p([-c_i^m,c_i^m])}^p&\leq 
C
\de_i^{-2p+1}
\ell_i^{1-p} \int  (\rho_i^{-1}\abs{\tauE(u_i)})^p ds d\th
+C E(\hat u_i)^\half (\int \vartheta_i^p ds)^\half \\
&+C\de_i^{-2p+1} \int \rho_i^{-(2p-1)} \vartheta_i^{\frac{3p}2}  ds 
+C\de_i^{-(p-1)} \int \rho_i^{-(p-1)} \abs{u_{s\theta}}^p ds d\theta\\
& =:C(T_i^{(1)}+T_i^{(2)}+T_i^{(3)}+T_i^{(4)})
\eeqas
and we will show that each of these terms tends to zero as $i\to \infty$. 

First of all, the relation \eqref{est:relation-tension} between the euclidean and the hyperbolic tension
and our main assumption \eqref{def:eps-1} imply that 
\beqs
T_i^{(1)}\leq \de_i^{-2p+1}
\ell_i^{1-p} (2X(\ell_i))^{1-\frac{p}2}\norm{\tau_{g_i}(u_i)}_{L^2(\Col(\ell_i))}^p\leq C\de_i^{-2p+1} \eps_i^p\leq C\de_i\to 0.
\eeqs
Since the endpoints of our interval $[b_i^{m-1},a_i^m]$ have distance at least $4\abs{\log\ell_i)}-\La_0$ from the high energy set   we can furthermore use that Lemmas \ref{lemma:ang-energy-standard} and \ref{lemma:H2-angular-better} as well as 
\eqref{est:F-by-F-hyp} to bound
\beq
\label{est:summer}
\vartheta_i(s)+\int_{\Cyl_1(s)} \abs{\partial_{s\theta} u_{i}}^2 \leq C \rho_i^2 \ell_i \eps_i^2 
+C\ell_i^2 e^{-\min(s-b_i^{m-1},a_i^{m}-s)} \text{ for all } s\in [b_i^{m-1},a_i^m].\eeq
As $E(\hat u_i)\leq C E(u_i)\leq C$ this immediately implies that $T_i^{(2)}\to 0$ while also
\beqas
T_i^{(3)}&\leq C\de_i^{-2p+1} \ell_i^{\frac{3p}{2}} \eps_i^{3p} \int\rho_i^{p+1}+C \de_i^{-2p+1} \ell_i^{p+1}\to 0.
\eeqas  
Similarly \eqref{est:summer} implies that
\beqs 
T_i^{(4)}\leq C\de_i^{-(p-1)}\eps_i^p \ell_i^{\frac{p}2}
\int \rho_i+C  \de_i^{-(p-1)} \ell_i\leq C  \ell_i^{\frac{p}2}\abs{\log\ell_i} +C\de_i^{-1}\ell_i\to 0
\eeqs
where we use in the penultimate step that $\int_{-X(\ell)}^{X(\ell)}\rho_\ell ds\leq C \abs{\log\ell}$.

We hence conclude that 
$\norm{\tau(v_i)}_{L^p([-c_i^m,c_i^m])}\to 0 \text{ for every } p\in [1,2]$ as claimed in Theorem \ref{thm:1}. 

After passing to a subsequence we can now assume that $c_i^m\to c\in [0,\infty]$.
If $c=0$ then the curves simply collapse to a point so there is nothing to show. 

If instead $c\in (0,\infty)$ then setting  $\tilde v_i=v_i(\frac {c_i^m}{c} \cdot)$ 
we obtain curves in $N$ that are parametrised over a fixed interval $I= [- c, c]$ with constant velocity $\la_i=\frac {c_i^m}{c}\to 1$ which satisfy
\beq
\label{est:v''}
\tilde v_i''=-A(\tilde v_i)(\tilde v_i',\tilde v_i')+g_i
\eeq for functions $g_i=\tau(\tilde v_i)$ with $\norm{g_i}_{L^2(I)}=\la_i \norm{\tau(v_i)}_{L^2([-c_i^m, c_i^m])}\to 0$. This implies in particular that the sequence $\tilde v_i$ is bounded in $W^{2,2}(I)$ and hence that it subconverges to a limit $v_\infty$ weakly in $W^{2,2}(I)$ and strongly in $C^{1,\al}(I)$ for $\al<\half$.
Hence the right hand side of \eqref{est:v''} converges strongly in $L^2$ to $- A(v_\infty)(\na v_\infty,\na v_\infty)$  which allows us to deduce  that $\tilde v_i$ converges indeed strongly in $W^{2,2}$ to the limit $v_\infty$ which satisfies
$\tauE(v_\infty)=v_\infty''+A(v_\infty)(v_\infty,v_\infty)=0$, i.e. is a geodesic.

Similarly, if $c_i^m\to \infty$ then we can apply the above argument locally on $[0,\infty)$ to conclude that the curves $v_i^m(\cdot-c_i^m)$ and $v_i^m(c_i^m-\cdot)$ subconverge strongly in $W^{2,2}_{loc}$ to infinite length geodesics. 

\section{Almost harmonic maps from degenerating tori} \label{sect:tori}
We recall that the moduli space of the torus is given by the quotients of $(\C,g_E)$ with respect to the lattices that are generated by $2\pi$ and $A+iB$ for $A\in(-\pi,\pi]$ and $B>0$ with $\abs{A+iB}\geq 2\pi$ (with strict inequality for $A<0$). 

We can thus view maps $u_i$ from degenerating unit area tori $(T^2,g_i)$ as maps from $\R\times S^1$ equipped with metric $g_i=\rho_i^2 g_E$ for $\rho_i=(2\pi B_i)^{-\half}$ with periodicity 
$u_i(s,\theta) =u_i( s+B_i,\theta+A_i)$. 

Given any closed curve $\al\in C^2(S^1,N)$ we can hence consider the maps $u_i(s,\theta)=\al(\frac{2\pi s}{B_i})$ whose tension decays according to 
\beq
\norm{\tau_{g_i}(u_i)}_{L^2(T^2,g_i)}^2\leq C \int_0^{B_i}\rho_i^{-2} B_i^{-4}=C B_i^{-2}=C \ell_i^4
\eeq
to see that we cannot expect almost harmonic maps to converge to \textit{geodesic} bubble trees if the tension decays no faster than $\inj(M,g_i)^2$. Here $\ell_i$ denotes the length of the shortest closed geodesic in the domain $(T^2,g_i)$ which is given by $\ell_i=2\inj(M,g_i)=2\pi \rho_i$ as $\abs{A_i+i B_i}\geq 2\pi$. 

The proof of our main result in the hyperbolic case now carries over to the new setting of maps from tori with only minor modifications and indeed simplifies since the conformal factors $\rho_i=\frac{\ell_i}{2\pi}$ are constant. 

To explain the necessary modifications we first note
that $\norm{\tau_{g_i}(u_i)}_{L^2(T^2,g_i)} =o(\ell_i^{2})$ is equivalent to asking that $\norm{\tau_{g_E}(u_i)}_{L^2([0,B_i]\times S^1,g_E)}=\rho_i \norm{\tau_{g_i}(u_i)}_{L^2([0,B_i]\times S^1,g_i)} =o(\ell_i^3)$. 

Defining the intervals $I_i^m$ now so that $\dist(I_i^m,\{s_i^{m-1},s_i^m\})= 8\abs{\log\ell_i}$ we then need to distinguish between the case where either
 $I_i^m=\emptyset$ or 
$ \ell_i^{-2}\sup_{I_i^m} \al_i\to 0$ and the case where $\sup_{I_i^m}\ell_i^{-2}\al_i\geq c_0>0$. In the former situation we again set 
$b_i^{m-1}=a_i^{m}= \half (s_i^{m-1}+s_i^m)$ so end up with a trivial connecting cylinder, while in the later case we choose $a_i^m$ and $b_i^{m-1}$ as the maximal and minimal elements of $I_i^m$ with $\ell_i^{-2}\al(s)\geq \de_i$ where we again set $\de_i=\min(\eps_i,\ell_i)^\half$, now for $\eps_i:= \ell_i^{-3}\norm{\tau_{g_E}(u_i)}_{L^2([0,B_i]\times S^1,g_E)}\to 0$. 
As $B_i\leq C\ell_i^{-2}$ we hence again get that
$\int_{I_i^m\setminus [b_i^{m-1},a_i]}\al_i\to 0$ and we can argue exactly as in the previous proof to obtain \eqref{claim:no-energy-loss-infty} and \eqref{claim:no-neck-infty}, i.e. to exclude the loss of energy or formation of necks on the extended bubble regions.  

Integrating \eqref{claim:est-osc-alpha} over the interval $[b_i^{m-1},a_i^m]$ whose length is bounded by $C\ell_i^{-2}$ 
and using that  $R_{u_i}\leq C\ell_i^8+C\eps_i^2 \ell_i^6$ on $I_i^m$ then allows us to conclude that $\al_i\geq (1-o(1)) \de_i\ell_i^2 $ on this interval and hence that also $\abs{\hatumi'}\geq  (1-o(1)) \de_i\ell_i^2 $ on the connecting cylinder. This allows us to carry out the rest of the proof exactly as in the hyperbolic case to complete the proof of Theorem \ref{thm:torus}.

M. Rupflin: Mathematical Institute, University of Oxford, Oxford OX2 6GG, UK\\
\textit{melanie.rupflin@maths.ox.ac.uk}

\end{document}